\documentclass{amsart}

\usepackage{amsmath,amssymb,amsfonts,enumerate,amsthm,graphicx,color,hyperref}
\usepackage{tikz, float} \usetikzlibrary {positioning}

\newcommand{\Mon}{\mbox{Mon}\,}

\newcommand{\Supp}{\mbox{Supp}\,}

\newcommand{\astab}{\mbox{astab}\,}

\newcommand{\Ind}{\mbox{Ind}\,}

\newcommand{\T}{\mathrm}

\newtheorem{thm}{Theorem}[section]
\newtheorem{cor}[thm]{Corollary}
\newtheorem{lem}[thm]{Lemma}

\newtheorem{defn}[thm]{Definition}

\newtheorem{rems}[thm]{Remarks}

\numberwithin{equation}{section}
\tikzstyle{Cgray}=[draw=black, scale = .4,circle, fill = white, minimum size=7mm]
\tikzstyle{Cwhite}=[scale = .8,circle, fill = white, minimum size=8mm]
\tikzstyle{Cblack}=[scale = .3,circle, fill = black, minimum size=3mm]
\tikzstyle{C0}=[scale = .9,circle, fill = black!0, inner sep = 0pt, minimum size=3mm]
\tikzstyle{C1}=[scale = .7,circle, fill = black!0, inner sep = 0pt, minimum size=3mm]
\tikzstyle{CW}=[scale = .7,circle, fill = white!0, inner sep = 0pt, minimum size=7mm]
\tikzstyle{Cg}=[scale = .8,circle, fill = gray, minimum size=8mm]

\begin{document}
\bibliographystyle{amsplain}

\title[A family of monomial ideals with the persistence property]{A family of monomial ideals with the persistence property}
\author[S. Moradi, M. Rahimbeigi, F. Khosh-Ahang, A. Soleyman Jahan]{Somayeh Moradi, Masoomeh Rahimbeigi, Fahimeh Khosh-Ahang and Ali Soleyman Jahan}
\address{Somayeh Moradi, Department of Mathematics, School of Science, Ilam University, P.O.Box 69315-516, Ilam, Iran.}
 \email{somayeh.moradi1@gmail.com}

\address{Masoomeh Rahimbeigi, Department of Mathematics, University of Kurdistan, Post Code
66177-15175, Sanandaj, Iran.}
\email{rahimbeigi$_{-}$masoome@yahoo.com}

\address{Fahimeh Khosh-Ahang, Department of Mathematics, School of Science,
Ilam University, P.O.Box 69315-516, Ilam, Iran.}
\email{fahime$_{-}$khosh@yahoo.com}

\address{Ali Soleyman Jahan, Department of Mathematics, University of Kurdistan, Post Code
66177-15175, Sanandaj, Iran.}
\email{solymanjahan@gmail.com}


\begin{abstract}
\noindent
In this paper we introduce a family of monomial ideals with the persistence property. Given positive integers $n$ and $t$, we consider the monomial ideal $I=Ind_t(P_n)$ generated by all monomials $\textbf{x} ^F$, where $F$ is an independent set of vertices of the path graph $P_n$ of size $t$, which is indeed the facet ideal of the $t$-th skeleton of the independence complex of $P_n$. We describe the set of associated primes of all powers of $I$ explicitly. It turns out that any such ideal $I$ has the persistence property. Moreover the index of stability of $I$ and the stable set of associated prime ideals of $I$ are determined.
\end{abstract}

\maketitle

\section*{Introduction}
Brodmann \cite{B} showed that for an ideal $I$ in a Noetherian ring $R$,  there exists a positive integer $ k_0$  such that $ \T{Ass} (I^{k}) = \T{Ass} (I^{k_0}) $ for all $  k\geq k_0 $, where $\T{Ass}(I)$ denotes the set of associated prime ideals of $R/I$.
A minimal such $k_0$ is called the \textit{index of stability} of $ I $ and $\T{Ass}(I^{k_0})$ is called the \textit{stable set of associated prime ideals} of $I$.
Although it is known that the sets $\T{Ass} (I^{k})$ stabilize for large $k$, their behavior for small values of $k$ may be irregular. Indeed a prime ideal in $\T{Ass} (I^{k+1})$ may not belong to $\T{Ass} (I^{k})$. An ideal $I$ is said to satisfy the \textit{persistence property} if  $\T{Ass}(I^{k})\subseteq \T{Ass} (I^{k+1})$ for any $k\geq 1$. Characterizing ideals with the persistence property, computing the stable set of associated prime ideals of a graded ideal in a polynomial ring $R$ and finding upper bounds for the index of stability of $I$ depending only on $R$ are some widely open questions in the context of Brodmann's theorem.  Although having the persistence property is a highly
desirable property, little is known about the ideals satisfying this property. Even for squarefree monomial ideals this property does not hold in general. Indeed H\`{a} and Sun \cite{HS}, presented a family of squarefree monomial ideals whose associated prime ideals of their powers do not form an ascending chain. In this regard finding classes of monomial ideals with the persistence property is of great interest. Few classes of monomial ideals are known to possess it. Edge ideals of graphs (\cite{MMV}), cover ideals of perfect graphs, cover ideals of clique graphs (\cite{FHV1}) and polymatroidal ideals (\cite{HRV}) are some families of ideals with this property.

In this paper we introduce a family of squarefree monomial ideals associated to path graphs and determine the associated primes of all powers of these ideals explicitly and show that these ideals satisfy the persistence property. Also we obtain the index of stability and the stable set of associated primes of such ideals precisely and characterize torsion-free ideals of this type.

Let $R=K[x_{1},\ldots, x_{n}]$ be a polynomial ring over a field $ K $ and $P_n:x_1,\ldots,x_n$ be a path graph on $n$ vertices. For any positive integer $t$, we consider the ideal
$$I=\Ind_t(P_n)=\langle\textbf{x} ^F:\ F\ \textit{is an independent set of vertices of the path graph}\ P_n\ \textit{of size}\ t\rangle,$$
where $\textbf{x} ^F=\prod_{x_i\in F}x_i$.
The main result of this paper is the following theorem which determines $\T{Ass}(I^k)$ explicitly for any non-zero ideal $I$ of this form and any positive integer $k$.

\begin{thm}\label{main}
Let $n, t>1$ and $k$ be positive integers such that $n\geq 2t-1$.
\begin{itemize}
\item[(i)] If $n=2t-1$, then $$\T{Ass}(I^k)=\T{Ass}(I)=\{\langle x_1\rangle,\langle x_3\rangle,\ldots,\langle x_{2t-1}\rangle\}.$$
\item[(ii)] If $n=2t$, then $$\T{Ass}(I^k)=\T{Ass}(I)=\{\langle x_{i_1},x_{i_2}\rangle:\ i_1<i_2,\ \textit{$i_1$ is odd and $i_2$ is even}\}.$$
\item[(iii)] If $n>2t$, then
\begin{multline*}
 \T{Ass}(I^k)=\{\langle x_{i_1},x_{i_2},\ldots,x_{i_{n-2t+2\ell}}\rangle:\ i_1<i_2<\cdots<i_{n-2t+2\ell}, \\
1\leq \ell \leq \min\{t,k\},\ \forall \ 1\leq j\leq  n-2t+2\ell,\ i_j\ \textit{and}\ j\ \textit{have the same parity}\}.
\end{multline*}
\end{itemize}
\end{thm}

Theorem \ref{main} implies that the ideal $I$ has the persistence property. Moreover,

$$\astab(I)=\left\{
  \begin{array}{ll}
    1 & \hbox{if}\ n=2t-1 \ \textrm{or}\ 2t \\
    t & \hbox{if}\ n>2t.
  \end{array}
\right.$$

and the stable set of associated prime ideals of $I$ is determined (see Corollary \ref{cor1}).



\section{A family of monomial ideals with the persistence property associated to path graphs}

In this section we study the set of associated primes of all powers of a family of squarefree monomial ideals associated to a path graph $P_n$ which will be defined in Definition \ref{def1}. First we give some terminology  and concepts which are needed in the sequel.

Throughout this paper $R=K[x_1,\ldots,x_n]$ is a polynomial ring over a field $K$ and $P_n:x_1,\ldots,x_n$ denotes a path graph with the vertex set $\{x_1,\ldots,x_n\}$ and the edge set $\{\{x_1,x_2\},\{x_2,x_3\},\ldots,\{x_{n-1},x_n\}\}$.

For any subset $F=\{x_{i_1}, \dots, x_{i_r}\}\subseteq\{x_1,\ldots,x_n\}$ we denote the monomial $x_{i_1}\cdots x_{i_r}$ in $R$ by $\textbf{x}^F $. The set of all monomials of $R$ is denoted by $\T{Mon}(R)$. For $u\in \T{Mon}(R)$, we set  $\Supp(u)=\{x_i:\ x_i|u\}$. Also if $u=x_1^{\alpha_1}\cdots x_n^{\alpha_n}$, we set $\deg_u(x_i)=\alpha_i$.
The unique minimal set of monomial generators of a monomial ideal $I$  is denoted by  $\mathcal{G}(I)$. Also the set of associated prime ideals of $R/I$ is denoted by $\T{Ass}(I)$.
For every two monomials $u$ and $v$,  the greatest common divisor of $u$ and $v$ is denoted by $\gcd(u,v)$. Also for our convenience, for every two monomials $u$ and $v$ we denote $\frac{u}{\T{gcd}(u,v)}$ by $u:v$.

For a simple graph $G$, a subset $A\subseteq V(G)$ is called an  \textit{independent set} of $G$ if no edge of $G$ is contained in $A$. For any subset $W\subseteq V(G)$, by $G[W]$ we mean the induced subgraph of $G$ on the
vertex set $W$, that is a subgraph of $G$ with the vertex set $W$ and its edge set consists of those edges of $G$ which are contained in $W$. The set $\{1,\ldots,n\}$ is denoted by $[n]$.

\begin{defn}
{\rm
An ideal $I$ in $R$ is said to satisfy the persistence property if $ \T{Ass}(I)\subseteq \T{Ass}(I^2)\subseteq \cdots\subseteq \T{Ass}(I^k)\subseteq\cdots$.  Also $I$ is called normally-torsion free if $ \T{Ass}(I^k)=\T{Ass}(I)$ for any integer $ k\geq 1$.
} \end{defn}
 For an ideal $I$ in $R$ the index of stability of $ I $ is denoted by $\astab(I)$ and the stable set of associated prime ideals of $I$ is denoted by $\T{Ass}^{\infty}(I)$.

\begin{defn}\label{def1}
{\rm
Let $n$ and $t$ be positive integers. The ideal $\Ind_t(P_n)$ in $K[x_1,\ldots,x_n]$ is defined as
$$\Ind_t(P_n)=\langle\textbf{x} ^F : \ F \ \textit{is an independent set of $P_n$ of size t}\rangle. $$
}
\end{defn}
Note that $\Ind_t(P_n)$ is indeed the facet ideal of the $t$-th skeleton of the independence complex of $P_n$ or the $t$-clique ideal of the complement of $P_n$ (see \cite{Mo}).

Through this section we assume that $I=\Ind_t(P_n)$ for some given positive integers $n$ and $t$.

Note that if $t=1$, then $I=\langle x_1, \dots, x_n\rangle$  and then $ \T{Ass}(I^k)=\T{Ass}(I)=\{\langle x_1, \dots, x_n\rangle\}$ for any integer $ k\geq 1$. Also, one can easily check that $I\neq 0$ if and only if $n\geq 2t-1$. Therefore in the sequel we may assume that $t>1$ and  $n\geq 2t-1$.

\begin{rems}\label{rem1}
{\rm
\begin{itemize}
\item[(i)] Let $t$ be a positive integer. Then any maximum independent set of $P_{2t}$ has $t$ elements.

\item[(ii)] Let $n=2t$. Then one can easily check that any minimal generator of $I$ is of the form $uv$, where $u=x_1x_3\cdots x_{2i-1}$ and $v=x_{2i+2}x_{2i+4}\cdots x_{2t}$ and $0\leq i\leq t$ is an integer. Indeed
$$I=\langle x_1x_3\cdots x_{2t-1}, x_1x_3\cdots x_{2t-3}x_{2t},x_1x_3\cdots x_{2t-5}x_{2t-2}x_{2t}, \ldots , x_1x_4x_6\cdots x_{2t}, x_2x_4\cdots x_{2t} \rangle.$$
\end{itemize}
Hereafter, we set $g_i=x_1x_3\cdots x_{2i-1}x_{2i+2}x_{2i+4}\cdots x_{2t}$ for any $0\leq i\leq t$.
}\end{rems}

\begin{lem}\label{lem1}
Let $n,t>1$ be integers with $n\geq 2t$, $1\leq \ell \leq t$, $0=i_0<i_1<i_2<\cdots<i_{n-2t+2\ell}<i_{n-2t+2\ell+1}=n+1$ and for any $\ 1\leq j\leq  n-2t+2\ell$, $i_j$ and $j$ have the same parity.  Set $B=\{x_{i_1},x_{i_2},\ldots,x_{i_{n-2t+2\ell}}\}$ and $A=\{x_1,\ldots,x_n\}\setminus B$. Then
\begin{itemize}
\item[(i)] $|A|=2t-2\ell$.

\item[(ii)] For any independent set $S$ of $P_n$ of size $t$, $|S\cap B|\geq \ell$.

\item[(iii)] Let  $H\subseteq B$ be a set of size $\ell$ such that for any $x_{i_j}\in H$, we have $x_{i_{j-1}}\notin H$ and $x_{i_{j+1}}\notin H$. Then $\textbf{{\rm x}}^{H}\textbf{{\rm x}}^{A}\in I$.
\end{itemize}
\end{lem}
\begin{proof}
$(\textbf{i})$ It is obvious.

$(\textbf{ii})$ Let $P_n[A]=L_{s_1}\cup\cdots\cup L_{s_r}$, where $L_{s_1},\ldots,L_{s_r}$ are connected components of $P_n[A]$ with $s_1,\dots, s_r$ vertices respectively. Note that $L_{s_m}$ is a path graph for any $m$. We have $\sum_{m=1}^r s_m=2t-2\ell$. For any $1\leq m\leq r-1$, we have $s_m=i_{j+1}-i_j-1$ for some $0\leq j\leq n-1$. Since $i_j$ and $j$ have the same parity, $i_j-j$ is even. Similarly $i_{j+1}-(j+1)$ is even. So $i_{j+1}-i_j=(i_{j+1}-(j+1))-(i_j-j)+1$ is odd and then $s_m=i_{j+1}-i_j-1$ is even for all $1\leq m\leq r-1$. Also $s_r=\sum_{m=1}^r s_m-\sum_{m=1}^{r-1} s_m=2t-2\ell-\sum_{m=1}^{r-1} s_m$ is even. For any $1\leq m\leq r$, let $s_m=2a_m$. Let $S$ be an independent set of $P_n$ of size $t$. By Remarks \ref{rem1}(i), $|S\cap V(L_{s_m})|\leq a_m$, for any
$1\leq m\leq r$. Thus $|S\cap A|\leq \sum_{m=1}^r a_m=\frac{1}{2}\sum_{m=1}^r s_m=t-\ell$. Since $|S|=t$, we have $|S\cap B|\geq \ell$ and this completes the proof.

$(\textbf{iii})$ Clearly $H$ is an independent set of $P_n$. Let $1\leq m\leq r$. For any connected component $L_{s_m}:x_{\lambda_m+1},\ldots,x_{\lambda_m+2a_m}$ of $P_n[A]$, if $\lambda_m+1 \neq 1$, then we have $x_{\lambda_m}=x_{i_{j_m}}$ for some $x_{i_{j_m}}\in B$ and if $\lambda_m+2a_m\neq n$, then we have $x_{\lambda_m+2a_m+1}=x_{i_{j_m+1}}\in B$. (Note that if $\lambda_m+1=1$, we should have $m=1$ and $\lambda_1=0$ and if $\lambda_m+2a_m=n$, we should have $m=r$ and $\lambda_r=n-2a_m$.) By assumption at least one of  $x_{i_{j_m}}$ and $x_{i_{j_m+1}}$ does not belong to $H$  for any $m$. For each $1\leq m\leq r$ we define $F_m$ as follows:\\
$\bullet$ If $x_{i_{j_m}}\notin H$, then set $F_m=\{x_{\lambda_m+1},x_{\lambda_m+3},\ldots,x_{\lambda_m+2a_m-1}\}$.\\
$\bullet$ If $x_{i_{j_m}}\in H$, then set $F_m=\{x_{\lambda_m+2},x_{\lambda_m+4},\ldots,x_{\lambda_m+2a_m}\}$ (note that $x_{i_{j_m+1}}\notin H$).\\

Clearly $F_m$  is an independent set of $P_n$ of size $a_m$ for any $1\leq m\leq r$ and  $H\cup F_1\cup\cdots \cup F_r$ is an independent set of $P_n$ of size $|H|+\sum_{m=1}^r a_m=\ell+(t-\ell)=t$. So $\textbf{x}^{H\cup F_1\cup\cdots \cup F_r}\in I$. Also since $H\cup F_1\cup\cdots \cup F_r\subseteq H\cup A$, we have $\textbf{x}^{H}\textbf{x}^{A}=\textbf{x}^{H\cup A}\in I$ and the claim is proved (note that $H\cap A=\emptyset$, since $H\subseteq B$).
\end{proof}

\begin{thm}\label{inclusion1}
Let $n, t>1$ and $k$ be positive integers.
\begin{itemize}
  \item[(i)]  If $n=2t$, then $$\{\langle x_{i_1},x_{i_2}\rangle:\ i_1<i_2,\ i_1\ \textit{ is odd and}\ i_2\ \textit{is even}\}\subseteq \T{Ass}(I^k).$$
  \item[(ii)] If $n>2t$, then
  \begin{multline*}
    \{\langle x_{i_1},x_{i_2},\ldots,x_{i_{n-2t+2\ell}}\rangle:\ i_1<i_2<\cdots<i_{n-2t+2\ell}, 1\leq \ell \leq \min\{t,k\}, \\
    \forall\ 1\leq j\leq  n-2t+2\ell,\ i_j\ \textit{and}\ j\ \textit{have the same parity}\}\subseteq \T{Ass}(I^k).
  \end{multline*}
\end{itemize}
\end{thm}

\begin{proof}

$(\textbf{i})$ Let $P=\langle x_{i_1},x_{i_2}\rangle$. Since $P$ is a prime ideal, in order to prove that $P\in \T{Ass}(I^k)$ it is enough to show that there exists $u\in \T{Mon}(R)$ such that $P=I^k:u$. Following the notations of Lemma \ref{lem1}, set $B=\{x_{i_1},x_{i_2}\}$ and $A=\{x_1,\ldots,x_n\}\setminus B$. Since $x_1x_3\cdots x_{2t-1} | x_{i_1}\textbf{x}^{A}$ and $x_2x_4\cdots x_{2t} | x_{i_2}\textbf{x}^{A}$, we have $u_1:=x_{i_1}\textbf{x}^{A}\in I$ and $u_2:=x_{i_2}\textbf{x}^{A}\in I$. Hence by setting $u=u_1^{k-1} \textbf{x}^A$ we have $x_{i_1}u=u_1^k\in I^k$ and $x_{i_2}u=u_1^{k-1}u_2\in I^k$. Therefore $x_{i_1},x_{i_2}\in I^k:u$ and so $P\subseteq I^k:u$.

Now, let $f\in I^k:u$ be a monomial (note that $I^k:u$ is a monomial ideal). Then $fu\in I^k$ and hence there exist $v_1,v_2,\ldots,v_k\in \mathcal{G}(I)$ such that $v_1v_2\cdots v_k| fu$. Use notations as in Lemma \ref{lem1} and set $B_i=\T{Supp}(v_i)\cap B$ for any $1\leq i\leq k$. By Lemma \ref{lem1}(ii), $|B_i|\geq 1$ for any $1\leq i\leq k$ (note that $\ell=1$ in this case). Clearly $\textbf{x}^{B_1}\textbf{x}^{B_2}\cdots \textbf{x}^{B_k}|v_1v_2\cdots v_k$. Thus $\textbf{x}^{B_1}\textbf{x}^{B_2}\cdots \textbf{x}^{B_k}| fu$. Since $B_i\subseteq B$ for all $1\leq i\leq k$, $A\cap B=\emptyset$ and $u=(x_{i_1}\textbf{x}^A)^{k-1}\textbf{x}^A$,  we have

\begin{equation}\label{eq}
\textbf{x}^{B_1}\textbf{x}^{B_2}\cdots \textbf{x}^{B_k}| fx_{i_1}^{k-1}.
\end{equation}

Since $|B_i|\geq 1$ for all $1\leq i\leq k$,  we have $\deg(\textbf{x}^{B_1}\textbf{x}^{B_2}\cdots \textbf{x}^{B_k})\geq k$. Thus considering relation (\ref{eq}), there exists $x_{i_j}\in B_m$ for some $1\leq m\leq k$ such that $x_{i_j}|f$. Note that $x_{i_j}\in B$ and hence $x_{i_j}\in P$. Thus $f\in P$ which completes the proof.

$(\textbf{ii})$ Let $n>2t$, $1\leq \ell \leq \min\{t,k\}$, $P=\langle x_{i_1},x_{i_2},\ldots,x_{i_{n-2t+2\ell}}\rangle$ such that $i_1<i_2<\cdots<i_{n-2t+2\ell}$ and $i_j$ and $j$ have the same parity for all  $j$. Then $n-2t+2\ell \geq 2\ell+1$. Set $B=\mathcal{G}(P)$ and $A=\{x_1,\ldots,x_n\}\setminus B$.
Since $P$ is a prime ideal, in order to prove that $P\in \T{Ass}(I^k)$, it is enough to show that there exists $u\in \T{Mon}(R)$ such that $P=I^k:u$.
For any $1\leq j\leq \ell-1$, set $u_{2j-1}=(x_{i_1}x_{i_3}x_{i_5}\cdots x_{i_{2\ell+1}})/x_{i_{2j-1}}$. Then $\T{Supp}(u_{2j-1})$ is a set of size $\ell$ which satisfies the assumption of Lemma \ref{lem1}(iii). Thus by Lemma \ref{lem1}(iii), $\textbf{x}^{A}u_{2j-1}\in I$ for any $1\leq j\leq \ell-1$.
Set $w=x_{i_1}x_{i_3}x_{i_5}\cdots x_{i_{2\ell-3}}$ and $$u=(\textbf{x}^Au_1)^{k-\ell+1}(\textbf{x}^Au_3)(\textbf{x}^Au_5)\cdots (\textbf{x}^Au_{2\ell-3})(\textbf{x}^Aw).$$

To show $I^k:u=P$, first let $f\in I^k:u$ be an arbitrary monomial. Then $fu\in I^k$. So there exist $v_1,v_2,\ldots,v_k\in \mathcal{G}(I)$ such that $v_1v_2\cdots v_k| fu$.  Set $B_i=\T{Supp}(v_i)\cap B$ for any $1\leq i\leq k$. By Lemma \ref{lem1}(ii), $|B_i|\geq \ell$ for any $1\leq i\leq k$. Clearly $\textbf{x}^{B_1}\textbf{x}^{B_2}\cdots \textbf{x}^{B_k}|v_1v_2\cdots v_k$. Thus $\textbf{x}^{B_1}\textbf{x}^{B_2}\cdots \textbf{x}^{B_k}| fu$.   Since $B_i\subseteq B$ for all $i$ and $A\cap B=\emptyset$, we have

\begin{equation}\label{eq1}
\textbf{x}^{B_1}\textbf{x}^{B_2}\cdots \textbf{x}^{B_k}| fu_1^{k-\ell+1}u_3u_5\cdots u_{2\ell-3}w.
\end{equation}

Since $|B_i|\geq \ell$ for all $i$, we have $\deg(\textbf{x}^{B_1}\textbf{x}^{B_2}\cdots \textbf{x}^{B_k})\geq k\ell$. Also
\begin{align*}
\deg(u_1^{k-\ell+1}u_3u_5\cdots u_{2\ell-3}w) &=(k-\ell+1)\deg(u_1)+\sum_{j=2}^{\ell-1} \deg(u_{2j-1})+\deg(w)\\
&=(k-\ell+1)\ell+(\ell-2)\ell+(\ell-1)\\
&=k\ell-1.
\end{align*}
Thus considering relation (\ref{eq1}), there exists $x_{i_j}\in B_m$ for some $1\leq m\leq k$ such that $x_{i_j}|f$. Note that $x_{i_j}\in B$ and hence $x_{i_j}\in P$. Thus $f\in P$ and so $I^k:u\subseteq P$.

Conversely, let $x_{i_m}\in P$ for some $1\leq m\leq n-2t+2\ell$. We should prove that $x_{i_m}u\in I^k$. We consider four cases:

\textbf{Case 1}. Let $m>2\ell -2$. Then $\{x_{i_m}\}\cup \T{Supp}(w)$ is a set of size $\ell$ satisfying the assumption of Lemma \ref{lem1}(iii). So by Lemma \ref{lem1}(iii), $\textbf{x}^{A} x_{i_m}w\in I$. Thus $$x_{i_m}u=(\textbf{x}^Au_1)^{k-\ell+1}(\textbf{x}^Au_3)(\textbf{x}^Au_5)\cdots (\textbf{x}^Au_{2\ell-3})(\textbf{x}^Ax_{i_m}w)\in I^k$$ (note that $\textbf{x}^{A}u_{2j-1}\in I$ for any $1\leq j\leq \ell-1$, as was mentioned above). Thus $x_{i_m}\in I^k:u$ in this case.

\textbf{Case 2}. Let $m=2\ell -2$. Then $x_{i_m}=x_{i_{2\ell-2}}$ and $\T{Supp}(w)\cup\{x_{i_{2\ell-2}},x_{i_{2\ell+1}}\}\setminus \{x_{i_{2\ell-3}}\}$ is a subset of $B$ of size $\ell$ satisfying the assumption of Lemma \ref{lem1}(iii). Thus $x^A(x_{i_{2\ell-2}}x_{i_{2\ell+1}}w/x_{i_{2\ell-3}})\in I$. Also $(\Supp(u_{2\ell-3})\cup \{x_{i_{2\ell-3}}\})\setminus \{x_{i_{2\ell+1}}\}$ is a subset of $B$ of size $\ell$ satisfying the assumption of Lemma \ref{lem1}(iii). Therefore  $\textbf{x}^A(x_{i_{2\ell-3}}u_{2\ell-3}/x_{i_{2\ell+1}})\in I$. Hence
\begin{multline*}
 x_{i_m}u=x_{i_{2\ell-2}}u=(\textbf{x}^Au_1)^{k-\ell+1}(\textbf{x}^Au_3)(\textbf{x}^Au_5)\cdots (\textbf{x}^Au_{2\ell-5})\\
 (\textbf{x}^Ax_{i_{2\ell-3}}u_{2\ell-3}/x_{i_{2\ell+1}})(\textbf{x}^Ax_{i_{2\ell-2}}x_{i_{2\ell+1}}w/x_{i_{2\ell-3}})\in I^k.
\end{multline*}
Thus $x_{i_m}\in I^k:u$.

\textbf{Case 3}. Let $m<2\ell -2$ and $m$ be an odd number. Then $m=2j-1$ for some $1\leq j\leq \ell-1$. By definition of $u_{2j-1}$, one can see that $x_{i_{2\ell+1}}|u_{2j-1}$. Also $(\Supp(u_{2j-1})\cup \{x_{i_{2j-1}}\})\setminus \{x_{i_{2\ell+1}}\}$ is a subset of $B$ of size $\ell$ satisfying the assumption of Lemma \ref{lem1}(iii). Thus by Lemma \ref{lem1}(iii), $\textbf{x}^A(x_{i_{2j-1}}u_{2j-1}/x_{i_{2\ell+1}})\in I$. Note that $\Supp(w)\cup\{x_{i_{2\ell+1}}\}$ is also a subset of $B$ of size $\ell$ satisfying the assumption of Lemma \ref{lem1}(iii) and then $\textbf{x}^Ax_{i_{2\ell+1}}w\in I$. Therefore
\begin{multline*}
  x_{i_m}u=x_{i_{2j-1}}u=(\textbf{x}^Au_1)^{k-\ell+1}(\textbf{x}^Au_3)(\textbf{x}^Au_5)\cdots (\textbf{x}^Au_{2j-3})\\
  (\textbf{x}^Ax_{i_{2j-1}}u_{2j-1}/x_{i_{2\ell+1}})(\textbf{x}^Au_{2j+1})\cdots (\textbf{x}^Au_{2\ell-3})(\textbf{x}^Ax_{i_{2\ell+1}}w)\in I^k.
\end{multline*}
Hence $x_{i_m}\in I^k:u$.\\
\textbf{Case 4}. Let $m<2\ell -2$ and $m$ be an even number. Then $m=2j\leq 2\ell-4$ for some $1\leq j\leq \ell-2$.
So $(\Supp(w)\cup\{x_{i_{2j}},x_{i_{2\ell-1}},x_{i_{2\ell+1}}\})\setminus\{x_{i_{2j-1}},x_{i_{2j+1}}\}$ is a subset of $B$ of size $\ell$ satisfying the assumption of Lemma \ref{lem1}(iii). Then by Lemma \ref{lem1}(iii), $$\textbf{x}^A(wx_{i_{2j}}x_{i_{2\ell-1}}x_{i_{2\ell+1}})/(x_{i_{2j-1}}x_{i_{2j+1}})\in I.$$
 Also $(\Supp(u_{2j-1})\cup \{x_{i_{2j-1}}\})\setminus \{x_{i_{2\ell-1}}\}$ and $(\Supp(u_{2j+1})\cup \{x_{i_{2j+1}}\})\setminus \{x_{i_{2\ell+1}}\}$ are subsets of $B$ of size $\ell$ satisfying the assumption of Lemma \ref{lem1}(iii). Thus $\textbf{x}^A(x_{i_{2j-1}}u_{2j-1}/x_{i_{2\ell-1}})\in I$ and $\textbf{x}^A(x_{i_{2j+1}}u_{2j+1}/x_{i_{2\ell+1}})\in I$.
Therefore
\begin{multline*}
 x_{i_m}u=x_{i_{2j}}u=(\textbf{x}^Au_1)^{k-\ell+1}(\textbf{x}^Au_3)(\textbf{x}^Au_5)\cdots (\textbf{x}^Au_{2j-3})(\textbf{x}^Ax_{i_{2j-1}}u_{2j-1}/x_{i_{2\ell-1}}) \\
(\textbf{x}^Ax_{i_{2j+1}}u_{2j+1}/x_{i_{2\ell+1}})(\textbf{x}^Au_{2j+3})\cdots (\textbf{x}^Au_{2\ell-3}) (\textbf{x}^Ax_{i_{2j}}x_{i_{2\ell-1}}x_{i_{2\ell+1}}w/x_{i_{2j-1}}x_{i_{2j+1}})\in I^k.
\end{multline*}
So $x_{i_m}\in I^k:u$.

Therefore $P\subseteq I^k:u$. So $I^k:u=P$ and $P\in \T{Ass} (I^k)$ as desired.
\end{proof}

Although in Theorem \ref{inclusion1}(i) we showed that when $n=2t$, $$\{\langle x_{i_1},x_{i_2}\rangle:\ i_1<i_2,\ i_1\ \textit{ is odd and}\ i_2\ \textit{is even}\}\subseteq \T{Ass}(I^k),$$
in the following theorem we show that indeed the equality holds. In fact we present a primary decomposition of $I^k$ for any positive integer $k$.

\begin{thm}\label{2t}
Suppose that $n, t>1$ and $k$ be positive integers such that $n=2t$. Then $$\T{Ass}(I^k)=\T{Ass}(I)=\{\langle x_{i_1},x_{i_2}\rangle:\ i_1<i_2,\ \textit{$i_1$ is odd and $i_2$ is even}\}.$$
\end{thm}
\begin{proof}
Set $$J=\bigcap_{r=1}^k \bigcap_{(i,j)}\langle \{x_i^r,x_j^{k+1-r}:\ i<j,\  \textit{$i$ is odd and $j$ is even}\}\rangle.$$ We claim that $I^k=J$.
Let $i$ be an odd number and $j$ be an even number such that $i<j$ and $1\leq r\leq k$ be an integer. We show that $I^k\subseteq \langle x_i^r,x_j^{k+1-r}\rangle$. Let $f\in I^k$ be an arbitrary monomial. Then there exists $u_1,\ldots,u_k\in \mathcal{G}(I)$ such that $u_1\cdots u_k|f$. If $x_i^r|u_1\cdots u_k$, then $x_i^r|f$ and hence $f\in \langle x_i^r,x_j^{k+1-r}\rangle$. If $x_i^r\nmid u_1\cdots u_k$, then without loss of generality we may assume that $x_i\nmid u_1,x_i\nmid u_2,\ldots,x_i\nmid u_{k-r+1}$. So by Remarks \ref{rem1}(ii), $x_j|u_1,x_j|u_2,\ldots,x_j| u_{k-r+1}$, since $i<j$ and $j$ is even. Thus $x_j^{k+1-r}|u_1\cdots u_k$ and hence $x_j^{k+1-r}|f$. So $f\in \langle x_i^r,x_j^{k+1-r}\rangle$. Thus $I^k\subseteq \langle x_i^r,x_j^{k+1-r}\rangle$. Since $i, j$ and $r$ were arbitrary, we have $I^k\subseteq J$. Now, we show that $J\subseteq I^k$.
We have

$$\begin{array}{rl}
  J= & \bigcap_{r=1}^k \bigcap_{odd\ i}  (\langle x_{i}^r,x_{i+1}^{k+1-r}\rangle
\cap \langle x_{i}^r,x_{i+3}^{k+1-r}\rangle\cap \cdots\cap\langle x_{i}^r,x_{2t}^{k+1-r}\rangle) \\
  = & \bigcap_{r=1}^k \bigcap_{odd\ i} \langle x_{i}^r, ( x_{i+1}x_{i+3}\cdots x_{2t})^{k+1-r}\rangle\\
  = & \bigcap_{odd\ i}\ \bigcap_{r=1}^k \langle x_{i}^r,(x_{i+1}x_{i+3}\cdots x_{2t})^{k+1-r}\rangle \\
  = & \bigcap_{odd\ i}\langle x_{i}^m(x_{i+1}x_{i+3}\cdots x_{2t})^{k-m}:\ 0\leq m\leq k\rangle.
\end{array}$$

Let $f\in J$ be an arbitrary monomial. Then by the above equalities, for any odd integer $1\leq i\leq 2t-1$, there exists an integer $0\leq m_i\leq k$ such that
\begin{equation}\label{mid}
 x_{i}^{m_i}(x_{i+1}x_{i+3}\cdots x_{2t})^{k-m_i}|f.
\end{equation}
Set $m'_i=\min\{m_1,m_3,\ldots,m_i\}$ for any odd integer $1\leq i\leq 2t-1$.
Then clearly $m'_i\leq m_i$ and $m'_1\geq m'_3\geq\cdots\geq m'_{2t-1}$. By relation (\ref{mid}), for any odd integer $i$, if $j\leq i$ and $j$ is odd, then we have $x_{i+1}^{k-m_j}|f$. So $x_{i+1}^{\max\{k-m_j:\ 1\leq j\leq i, j\ is\ odd \}}|f$. Note that $\max\{k-m_j:\ 1\leq j\leq i, j\ is\ odd \}=k-\min\{m_j:\ 1\leq j\leq i, j\ is\ odd\}=k-m'_i$. Therefore $x_{i+1}^{k-m'_i}|f$ for any odd integer $i$. Also by (\ref{mid}),
$x_i^{m_i}|f$ for any odd number $i$. So
\begin{equation}\label{mid2}
x_1^{m_1}x_3^{m_3}\cdots x_{2t-1}^{m_{2t-1}}x_2^{k-m'_1}x_4^{k-m'_3}\cdots x_{2t}^{k-m'_{2t-1}}|f.
\end{equation}

With the notations of Remarks \ref{rem1},
we have
 \begin{multline*}
   g_0^{k-m'_1}g_1^{m'_1-m'_3}g_2^{m'_3-m'_5}\cdots g_{t-1}^{m'_{2t-3}-m'_{2t-1}} g_t^{m'_{2t-1}} \\
  =x_1^{m'_1}x_3^{m'_3}\cdots x_{2t-1}^{m'_{2t-1}}x_2^{k-m'_1}x_4^{k-m'_3}\cdots x_{2t}^{k-m'_{2t-1}}.
 \end{multline*}
Since $m'_i\leq m_i$ for any odd number $i$, using relation  (\ref{mid2}), we have $$g_0^{k-m'_1}g_1^{m'_1-m'_3}g_2^{m'_3-m'_5}\cdots g_{t-1}^{m'_{2t-3}-m'_{2t-1}} g_t^{m'_{2t-1}}|\ f.$$

Also
\begin{multline*}
 g_0^{k-m'_1}g_1^{m'_1-m'_3}g_2^{m'_3-m'_5}\cdots g_{t-1}^{m'_{2t-3}-m'_{2t-1}} g_t^{m'_{2t-1}}\\
 \in I^{(k-m'_1)+(m'_1-m'_3)+\cdots+(m'_{2t-3}-m'_{2t-1})+m'_{2t-1}}=I^k.
\end{multline*}

Thus $f\in I^k$. So $J\subseteq I^k$. Thus
\begin{equation}\label{eqq1}
I^k=\bigcap_{r=1}^k \bigcap_{(i,j)}\langle x_i^r,x_j^{k+1-r}:\ i<j,\  \textit{$i$ is odd and $j$ is even}\rangle
\end{equation}
is a primary decomposition of $I^k$.
Also it is easy to see from \ref{eqq1} that $$\T{Ass}(I^k)=\{\langle x_{i_1},x_{i_2}\rangle:\ i_1<i_2,\ \textit{$i_1$ is odd and $i_2$ is even}\}.$$

\end{proof}

The following lemma is essential to prove Theorem \ref{main}.
\begin{lem}\label{lem2}
Assume that $n,t>1$ and $k$ are positive integers with $n\geq 2t$ and $P\in \T{Ass}(I^k)$. Suppose that $A=\{x_i: i\in [n], x_i\notin P\}$ and $P_n[A]=L_{s_1}\cup\cdots\cup L_{s_r}$, where $L_{s_1},\ldots,L_{s_r}$ are connected components of $P_n[A]$ with $s_1, \dots, s_r$ vertices respectively. Then $s_i$ is even for all $1\leq i\leq r$.
\end{lem}
\begin{proof}
Let $n\geq 2t$ and $P\in \T{Ass}(I^k)$. Then by \cite[Corollary 1.3.10]{HH} there exists $u\in \Mon(R)\setminus I^k$ such that
$P=I^k:u$. Note that $P$ is generated by some variables. If $P=\langle x_1, \dots, x_n\rangle$, then $s_i=0$ for all $i$ and there is nothing to prove. Let $P\neq \langle x_1, \dots, x_n\rangle$ and by contradiction assume that $s_i$ is odd for some $1\leq i\leq r$ and let $s_i=2m+1$. Let $L_{2m+1}:x_{j+1},x_{j+2},\ldots,x_{j+2m+1}$.

We consider three cases and in each case we get a contradiction.

\textbf{Case 1}. Let $j\neq 0$ and $j+2m+1\neq n$.

Then we have $x_j,x_{j+2m+2}\in P=I^k:u$. Thus there exist $u_1,\ldots,u_k\in \mathcal{G}(I)$ and $w\in \T{Mon}(R)$ such that $x_ju=u_1\cdots u_kw$ and there exist $u'_1,\ldots,u'_k\in \mathcal{G}(I)$ and $w'\in \T{Mon}(R)$ such that $x_{j+2m+2}u=u'_1\cdots u'_kw'$. Since $u\notin I^k$, we have
\begin{equation}\label{16}
x_j\nmid w \ \T{and} \  x_{j+2m+2}\nmid w'.
\end{equation}
So $x_j|u_1\cdots u_k$ and $x_{j+2m+2}|u'_1\cdots u'_k$. Three cases may happen. In each case we get a contradiction.

\textbf{Subcase 1}. There exists $1\leq q\leq k$ such that $x_j|u_q$ and $x_{j+2m+2}\nmid u_q$.

Without loss of generality assume that $q=1$. So $x_j|u_1$ and $x_{j+2m+2}\nmid u_1$. Note that $\Supp(u_1)$ is an independent set of $P_n$ of size $t$. Also $|\Supp(u_1)\cap \{x_j,x_{j+1},\ldots,x_{j+2m+1}\}|\leq m+1$, since any independent set of $P_n$ which is contained in $\{x_j,x_{j+1},\ldots,x_{j+2m+1}\}$ has at most $m+1$ elements. Set $$v_1=x_{j+1}x_{j+3}\cdots x_{j+2m+1}u_1/\gcd(u_1, \prod_{i=j}^ {j+2m+1}x_i).$$
Then
$$\Supp(v_1)=(\Supp(u_1)\setminus \{x_j,x_{j+1},\ldots,x_{j+2m+1}\})\cup \{x_{j+1},x_{j+3},\ldots,x_{j+2m+1}\}.$$
Since $x_{j+2m+2}\nmid u_1$, $\Supp(v_1)$ is an independent set of $P_n$. Also $|\Supp(v_1)|\geq t$, since at most $m+1$ elements has been removed from $\Supp(u_1)$ and exactly $m+1$ elements have been added to $\Supp(u_1)$ to get $\Supp(v_1)$. So $v_1\in I$. Set $v=v_1u_2u_3\cdots u_k$.
Then $v\in I^k$ and hence $v:u\in I^k:u=P$. Since $u=(u_1/x_j)u_2u_3\cdots u_k w$, we have $\gcd(v,u)=u_2u_3\cdots u_k\gcd(v_1,(u_1/x_j)w)$ and since $x_j\nmid v_1$, $\gcd(v_1,(u_1/x_j)w)=\gcd(v_1,u_1w)$.
Therefore
\begin{align*}
v:u&=v/\gcd(v,u)\\
&=v_1u_2u_3\cdots u_k/(u_2u_3\cdots u_k\gcd(v_1,u_1w))\\
&=v_1/\gcd(v_1,u_1w).
\end{align*}
We have $v_1/\gcd(v_1,u_1w)|\ v_1/\gcd(v_1,u_1)$.  It is easy to see that $$v_1/\gcd(v_1,u_1)|x_{j+1}x_{j+3}\cdots x_{j+2m+1}.$$ So $v:u\ |x_{j+1}x_{j+3}\cdots x_{j+2m+1}$. But $v:u\in P$, which means that there is $x_{\lambda}\in P$ such that $x_{\lambda}|v:u$, so $x_{\lambda}|x_{j+1}x_{j+3}\cdots x_{j+2m+1}$ which is impossible, since $x_{j+1},x_{j+3},\ldots ,x_{j+2m+1}\in A$ and hence are not in $P$. So in this case we get a contradiction.

\textbf{Subcase 2}. There exists $1\leq q\leq k$ such that $x_{j+2m+2}| u'_q$ and $x_j\nmid u'_q$.
By similar argument to Case 1, one can get a contradiction.

\textbf{Subcase 3}. For any $1\leq q\leq k$, if $x_j|u_q$, then $x_{j+2m+2}|u_q$ and if $x_{j+2m+2}| u'_q$, then $x_j| u'_q$.

This assumption  together with relation (\ref{16}) imply that

\begin{align*}
\deg _u(x_j)< \deg _{u}(x_j)+1= \deg _{x_ju}(x_j)\leq \deg _{x_ju}(x_{j+2m+2})=\deg _u(x_{j+2m+2}),
\end{align*}
and
\begin{multline*}
 \deg _u(x_{j+2m+2})<\deg _u(x_{j+2m+2})+1=\deg _{x_{j+2m+2}u}(x_{j+2m+2}) \\
 \leq \deg _{x_{j+2m+2}u}(x_{j})=\deg _u(x_j)
\end{multline*}

which is a contradiction.

\textbf{Case 2}. Let $j=0$. Then we have $2m+1<n$, since $P\neq (0)$. Also $L_{2m+1}:x_{1},x_{2},\ldots,x_{2m+1}$ is an induced path of $P_n[A]$, $x_{2m+2}\in P$ and $x_i\notin P$ for any $1\leq i\leq2m+1$. Hence there exist $u_1,\ldots,u_k\in \mathcal{G}(I)$ and $w\in \T{Mon}(R)$ such that
\begin{equation}\label{eq3}
x_{2m+2}u=u_1\cdots u_kw.
\end{equation}

Since $u\notin I^k$, we have $x_{2m+2}\nmid w$. Without loss of generality assume that $x_{2m+2}|u_1$. Set
\[u'_1=x_1x_3\cdots x_{2m+1}(u_1/\gcd(u_1,\prod_{i=1}^{2m+2} x_i)).
\]
Since $\Supp(u_1)$ is an independent set of $P_n$, we have $|\Supp(u_1)\cap\{x_i:\ 1\leq i\leq 2m+2\}|\leq m+1$.
Thus
 \[
 \Supp(u'_1)=(\Supp(u_1)\setminus \{x_i:\ 1\leq i\leq 2m+2\})\cup\{x_{1},x_{3},\ldots, x_{2m+1}\}
  \]
  is an independent set of $P_n$ of size at least $t$ and then $u'_1\in I$. Therefore
\[
u'_1u_2\cdots u_kw=x_1x_3\cdots x_{2m+1}(u_1/\gcd(u_1,\prod_{i=1}^{2m+2} x_i))u_2\cdots u_kw\in I^k.
 \]
Thus $x_1x_3\cdots x_{2m+1}(u_1/x_{2m+2})u_2\cdots u_kw\in I^k$, since $x_{2m+2}|\gcd(u_1,\prod_{i=1}^{2m+2} x_i)$. Therefore using \ref{eq3} we have $x_1x_3\cdots x_{2m+1}u\in I^k$ and hence
$x_1x_3\cdots x_{2m+1}\in I^k:u=P$. Since $P$ is a prime ideal, we have $x_i\in P$ for some $1\leq i\leq 2m+1$, which is a contradiction.

\textbf{Case 3}. Let $j+2m+1=n$. Since $P\neq (0)$, we have $j\neq 0$. Also $x_j\in P$ and $x_i\notin P$ for any $j+1\leq i\leq j+2m+1$.
There exist $u_1,\ldots,u_k\in \mathcal{G}(I)$ and $w\in \T{Mon}(R)$ such that
\begin{equation}\label{eq4}
x_ju=u_1\cdots u_kw.
\end{equation}
Since $u\notin I^k$, we have $x_j\nmid w$. Without loss of generality assume that $x_j|u_1$. Set
\[u'_1=x_{j+1}x_{j+3}\cdots x_{j+2m+1}(u_1/\gcd(u_1,\prod_{i=j}^{j+2m+1} x_i)).
\]
Since $\Supp(u_1)$ is an independent set of $P_n$, we have $|\Supp(u_1)\cap\{x_i:\ j\leq i\leq j+2m+1\}|\leq m+1$. Moreover $x_j\in \Supp(u_1)$.
 Thus
 \[
 \Supp(u'_1)=(\Supp(u_1)\setminus \{x_i:\ j\leq i\leq j+2m+1\})\cup\{x_{j+1},x_{j+3},\ldots, x_{j+2m+1}\}
  \]
  is an independent set of $P_n$ of size at least $t$ and then
  $u'_1\in I$. Therefore
\[
u'_1u_2\cdots u_kw=x_{j+1}x_{j+3}\cdots x_{j+2m+1}(u_1/\gcd(u_1,\prod_{i=j}^{j+2m+1} x_i))u_2\cdots u_kw\in I^k.
 \]
Thus $x_{j+1}x_{j+3}\cdots x_{j+2m+1}(u_1/x_j)u_2\cdots u_kw\in I^k$, since $x_j|\gcd(u_1,\prod_{i=j}^{j+2m+1} x_i)$. Therefore using \ref{eq4} we have $x_{j+1}x_{j+3}\cdots x_{j+2m+1}u\in I^k$ and hence
$x_{j+1}x_{j+3}\cdots x_{j+2m+1}\in I^k:u=P$. Since $P$ is a prime ideal, we have $x_i\in P$ for some $j+1\leq i\leq j+2m+1$, which is a contradiction.

\end{proof}

Now, we state the proof of Theorem \ref{main}. We remind that $I\neq 0$ if and only if $n\geq 2t-1$.

\textbf{Proof of Theorem \ref{main}}.
$\textbf{(i)}$ It is easy to see that if $n=2t-1$, then $I=\langle x_1x_3x_5\cdots x_{2t-1}\rangle$. So for any $k\geq1$, $I^k=\langle x_1^kx_3^kx_5^k\cdots x_{2t-1}^k\rangle$ and then $I^k=\langle x_1^k\rangle\cap\langle x_3^k\rangle\cap \cdots \cap \langle x_{2t-1}^k\rangle$ is a minimal primary decomposition of $I^k$. Therefore $$\T{Ass}(I^k)=\T{Ass}(I)=\{\langle x_1\rangle,\langle x_3\rangle,\ldots,\langle x_{2t-1}\rangle\}.$$

$\textbf{(ii)}$ It is the statement of Theorem \ref{2t}.

$\textbf{(iii)}$  Assume that $P\in \T{Ass}(I^k)$. Following the notation of Lemma \ref{lem2} let $A=\{x_i:  i\in [n], x_i\notin P\}$ and $s_1,\ldots,s_r$ be as defined there. First we show that $|A|\leq 2t-2$. By contradiction assume that $|A|\geq 2t-1$ and let $A=\{x_{j_1},x_{j_2},\ldots,x_{j_{2t-1}},\dots\}$ such that $j_1<j_2<\cdots$. Then $\{x_{j_1},x_{j_3},x_{j_5},\ldots,x_{j_{2t-1}}\}$ is an independent set of $P_n$ of size $t$. So we have $f=x_{j_1}x_{j_3}x_{j_5}\cdots x_{j_{2t-1}}\in \mathcal{G}(I)$. Thus $f^k:u\in I^k:u=P$. So there is $x_i\in P$ such that $x_i|\ f^k:u$.
Clearly $f^k:u| f^k=x_{j_1}^k x_{j_3}^k x_{j_5}^k\cdots x_{j_{2t-1}}^k$. So $x_i|x_{j_1}^k x_{j_3}^k x_{j_5}^k\cdots x_{j_{2t-1}}^k$. But $x_{j_1},x_{j_3},x_{j_5},\ldots,x_{j_{2t-1}}\notin P$, which is a contradiction. Thus $|A|\leq 2t-2$. On the other hand, in view of Lemma \ref{lem2}, $|A|=\sum_{i=1}^r s_i$ is even. Hence there is an integer $\ell\geq 1$ such that $|A|=2t-2\ell$. So we may assume that $P=\langle x_{i_1},x_{i_2},\ldots,x_{i_{n-2t+2\ell}}\rangle$ such that $i_1<i_2<\cdots<i_{n-2t+2\ell}$ for some integer $1\leq \ell \leq t$.

Now we are going to show that $i_j$ and $j$ have the same parity. To this aim, first we show that $i_1$ is odd. If $i_1$ is even, then $i_1>1$ and then $L_{s_1}:x_1,\ldots,x_{i_1-1}$. So $s_1=i_1-1$ and then $s_1$ is odd, which is a contradiction by Lemma \ref{lem2}. So $i_1$ is odd. Now, we show that $i_{j+1}-i_j$ is odd for any $j\geq 1$. If $i_{j+1}-i_j=1$, then we are done. So let $i_{j+1}-i_j>1$. Then $i_j+1<i_{j+1}$, $x_{i_j+1},x_{i_j+2},\ldots,x_{i_{j+1}-1}\notin P$ and $L:x_{i_j+1},x_{i_j+2},\ldots,x_{i_{j+1}-1}$ is a connected component of $P_n[A]$ and hence $L=L_{s_d}$ for some $d$. So $s_d=i_{j+1}-i_j-1$. But $s_d$ is even by Lemma \ref{lem2}, so $i_{j+1}-i_j$ is odd. Since $i_1$ is odd and for all $j\geq 1$, $i_{j+1}-i_j$ is odd, $i_j$ and $j$ have the same parity for all $j$ as desired.

So
\begin{multline*}
  P\in \{\langle x_{i_1},x_{i_2},\ldots,x_{i_{n-2t+2\ell}}\rangle:\ i_1<i_2<\cdots<i_{n-2t+2\ell},\\
   1\leq \ell \leq \min\{t,k\}, \forall\ 1\leq j\leq  n-2t+2\ell,\ i_j\ \textit{and}\ j\ \textit{have the same parity}\}.
\end{multline*}
Now, Theorem \ref{inclusion1} completes the proof.





As an immediate corollary of Theorem \ref{main}, we can obtain $\astab(I)$ and $\T{Ass}^{\infty}(I)$ explicitly.

\begin{cor}\label{cor1}
Let $n$ and $t>1$ be positive integers such that $n\geq 2t-1$. Then $I$ has the persistence property. Also
\begin{itemize}
  \item[(i)] If $n=2t-1$ or $n=2t$, then $\astab(I)=1$, hence $I$ is normally torsion-free and $\T{Ass}^{\infty}(I)=\T{Ass}(I)$.
  \item[(ii)]If $n>2t$, then $\astab(I)=t$ and
  \begin{multline*}
    \T{Ass}^{\infty}(I)=\{\langle x_{i_1},x_{i_2},\ldots,x_{i_{n-2t+2\ell}}\rangle:\ i_1<i_2<\cdots<i_{n-2t+2\ell},  \\
   1\leq \ell \leq t, \forall \ 1\leq j\leq  n-2t+2\ell,\ i_j\ \textit{and}\ j\ \textit{have the same parity}\}.
  \end{multline*}
\end{itemize}
\end{cor}

\providecommand{\bysame}{\leavevmode\hbox
to3em{\hrulefill}\thinspace}

\end{document}